\documentclass{amsart}

\usepackage[top=1in,bottom=1in,left=1in,right=1in]{geometry}
\usepackage{amsfonts,
amsmath,
amsthm,
amssymb,
url,
hyperref,
MnSymbol,
color}
\hypersetup{colorlinks=true,
urlcolor=blue,
linkcolor=cyan}

\newcommand{\C}{\ensuremath{{\mathbb{C}}}}
\newcommand{\F}{\ensuremath{{\mathbb{F}}}}

\newcommand{\Q}{\ensuremath{{\mathbb{Q}}}}

\newcommand{\Z}{\ensuremath{{\mathbb{Z}}}}
\newcommand{\lv}{\ensuremath{\left\vert}}
\newcommand{\rv}{\ensuremath{\right\vert}}
\newcommand{\lp}{\ensuremath{\left(}}
\newcommand{\rp}{\ensuremath{\right)}}
\newcommand{\lb}{\ensuremath{\left\{}}
\newcommand{\rb}{\ensuremath{\right\}}}

\DeclareMathOperator{\Surj}{Surj}
\DeclareMathOperator{\Aut}{Aut}
\DeclareMathOperator{\sub}{sub}
\DeclareMathOperator{\rank}{rank}

\theoremstyle{plain}
\newtheorem{clheur}{Heuristic}[section]
\newtheorem{subg4}[clheur]{Lemma}
\newtheorem{subg5}[clheur]{Lemma}
\newtheorem{converge3}[clheur]{Corollary}
\newtheorem{converge4}[clheur]{Corollary}
\newtheorem{q1}[clheur]{Corollary}
\newtheorem{qme}[clheur]{Corollary}
\newtheorem{autid}[clheur]{Lemma}
\newtheorem{ell2}[clheur]{Lemma}
\newtheorem{Sagain}[clheur]{Corollary}

\theoremstyle{definition}
\newtheorem{Not23}[clheur]{Notation}
\newtheorem{Not24}[clheur]{Definition}
\newtheorem{Not4}[clheur]{Definitions}
\newtheorem{Not6}[clheur]{Definition}
\newtheorem{Not7}[clheur]{Notation}
\newtheorem{Not71}[clheur]{Definitions}

\title{Some finite abelian group theory and some $q$-series identities}
\author{Derek Garton}
\date{May 22, 2014}

\begin{document}
\maketitle

\section{Introduction} \label{introduction}

In 1984, Cohen and Lenstra~\cite{CL} famously presented a family of conjectures about the structure of ideal class groups of number fields.
These conjectures follow from this heuristic:
\begin{clheur} \label{clheur}
For any odd prime $\ell$, a finite abelian $\ell$-group should appear as the $\ell$-Sylow subgroup of the ideal class group of an imaginary quadratic extension of $\Q$ with frequency inversely proportional to the order of its automorphism group.
\end{clheur}
Their heuristic leads to a  probability distribution on the poset of isomorphism classes of finite abelian $\ell$-groups.
They computed various averages on this poset using this distribution, averages that lead directly to their conjectures.
Along the way, they proved a variety of identities concerning finite abelian $\ell$-groups, such as the identity of Hall~\cite{Hall}, which shows that their heuristic does indeed give a probability distribution:
\[
\sum_{\substack{\text{finite abelian}\\
\ell\text{-groups }A}}{\frac{1}{\lv\Aut{A}\rv}}
=\prod_{i=1}^\infty{\lp1-\ell^{-i}\rp^{-1}}.
\]

More recently, Delaunay and Delaunay-Jouhet (\cite{D},\cite{DJ},\cite{DJ2}) computed new averages using this distribution.
They proceeded to use these averages to study the structure of Tate-Shafarevich groups and Selmer groups of elliptic curves.
Moreover, they showed that the averages they computed are consistent with the conjectures of Poonen-Rains (\cite{PR}).
Like Cohen and Lenstra, they discovered new and interesting identities along the way to computing these averages.

In this paper, we continue the tradition of studying finite abelian group identities to investigate number theoretic phenomena.
The simplest identity is \hyperref[subg4]{Lemma~\ref*{subg4}}: for any finite abelian $\ell$-group $A$ of $\ell$-rank $r$,
\[
\sum_{\substack{\text{finite abelian}\\
\ell\text{-groups }B\\
\rank_\ell{B}=r}}
{\frac{\lv\lb C\leq B\mid C\simeq A\rb\rv}{\lv\Aut{B}\rv}}
=\frac{1}{\lv\Aut{A}\rv}\cdot
\prod_{i=1}^r{\lp1-\ell^{-i}\rp^{-1}},
\]
but see also \hyperref[converge3]{Corollary~\ref*{converge3}}, \hyperref[ell2]{Lemma~\ref*{ell2}}, and \hyperref[Sagain]{Corollary~\ref*{Sagain}}.
These identities are crucial in constructing a random matrix justification for Malle's conjectural distribution of relative ideal class groups of number fields in the presence of unexpected roots of unity (see~\cite{Mal0} and~\cite{Mal}).
See~\cite{G} for this random matrix construction.

The organization of this paper is as follows.
\hyperref[notdef]{Section~\ref*{notdef}} contains notation and definitions, Sections~\ref{prelim} and~\ref{prelim2} present the averages of finite abelian $\ell$-groups, and \hyperref[qprelim]{Section~\ref*{qprelim}} presents two $q$-series identities which are important for later computations in this paper.

\section{Notation and definitions} \label{notdef}

Let $\ell$ be an odd prime and let $\mathcal{G}$ be the poset of isomorphism classes of finite abelian $\ell$-groups, with the relation $\left[A\right]\leq\left[B\right]$ if and only if there exists an injection $A\hookrightarrow B$.
(For notational simplicity, we will conflate finite abelian $\ell$-groups and the equivalence classes containing them.)
If $i\in\Z^{>0}$, let
\[
\rank_{\ell^i}{A}:=\dim_{\F_\ell}{\lp\ell^{i-1}A/\ell^iA\rp}.
\]
We will abbreviate $\rank_\ell{A}$ by $\rank{A}$.

\begin{Not6} \label{Not6}
For any $r_1,\ldots,r_i\in\Z^{\geq0}$, define the following subposet of $\mathcal{G}$:
\[
\mathcal{G}\lp r_1,\ldots,r_i\rp
:=\lb A\in\mathcal{G}\mid\rank_{\ell^j}(A)=r_j\text{ for all }1\leq j\leq i\rb.
\]
\end{Not6}

Next, we define a function $S:\mathcal{G}\times\mathcal{G}\to\Z$ that is related to the M\"{o}bius function on $\mathcal{G}$.
See~\cite{GMob} for a deeper discussion of this relationship.

\begin{Not4} \label{Not4}
For any $A,B\in\mathcal{G}$, let $\sub{(A,B)}$ be the number of subgroups of $B$ that are isomorphic to $A$.
If $A\in\mathcal{G}$, an \emph{$A$-chain} is a finite (possibly empty) linearly ordered subset of $\lb B\in\mathcal{G}\mid B>A\rb$.
Now, given an $A$-chain $\mathfrak{C}=\lb A_j\rb_{j=1}^i$, define
\[
\sub{\lp\mathfrak{C}\rp}
:=(-1)^i\sub{\lp A,A_1\rp}\prod_{j=1}^{i-1}{\sub{\lp A_j,A_{j+1}\rp}}.
\]
(We set $\sub{\lp\mathfrak{C}\rp}=1$ if $\mathfrak{C}$ is empty.)
Finally, for any $A,B\in\mathcal{G}$, let
\[
S(A,B)
=\begin{cases}
\hspace{18px}0&\text{if }A\nleq B,\\
\hspace{18px}1&\text{if }A=B,\\
\displaystyle{
\sum_{\substack{A\text{-chains }\mathfrak{C},\\
\max{\mathfrak{C}}=B}}
{\sub{\lp\mathfrak{C}\rp}}}&\text{if } A<B.
\end{cases}
\]
\end{Not4}

The definitions that follow are standard in the study of $q$-series.

\begin{Not7} \label{Not7}
For $q,z\in\C$ with $\lv q\rv<1$ and any $i\in\Z^{\geq0}$, let
\[
\lp z;q\rp_i:=\prod_{j=0}^{i-1}{\lp1-q^jz\rp}.
\]
To ease notation, set $(q)_i:=\lp q;q\rp_i$.
Next, we define the $q$-binomial coefficients: for any $k,m\in\Z^{\geq0}$, let
\[
\binom{k}{m}_q:=\frac{(q)_k}{(q)_m(q)_{k-m}},
\]
with $\binom{k}{m}_q:=0$ if $k<m$.
\end{Not7}

In Sections~\ref{prelim} and~\ref{prelim2}, we will set $q=1/\ell$, but the results in \hyperref[qprelim]{Section~\ref*{qprelim}} hold for any complex number $q$ of modulus less than 1.

\section{Some identities from finite abelian group theory} \label{prelim}

Before proving the first identity, we will follow~\cite{CL} and~\cite{Hall} in making two more definitions, which we will use only in the proof of \hyperref[subg4]{Lemma~\ref*{subg4}}.
In this section, we set $q=1/\ell$.

\begin{Not24} \label{Not24}
For any $A\in\mathcal{G}$ and $i\in\Z^{\geq0}$, define the \emph{$i$-weight of $A$} to be
\[
w_i(A):=\frac{\lv\Surj{\lp\Z^i,A\rp}\rv}{\lv A\rv^i\lv\Aut{A}\rv},
\]
and for any $j\in\Z^{\geq0}$, define
\[
w_i\lp\ell^j\rp:=\sum_{\substack{A\in\mathcal{G}\\
\lv A\rv=\ell^j}}{w_i(A)}.
\]
\end{Not24}

\begin{subg4} \label{subg4}
Suppose that $A\in\mathcal{G}$ and $\rank{A}=r$.
Then
\[
\sum_{B\in\mathcal{G}(r)}
{\frac{\sub{(A,B)}}{\lv\Aut{B}\rv}}
=\frac{1}{(q)_r\lv\Aut{A}\rv}.
\]
\end{subg4}
\begin{proof}
By Proposition~3.1 of~\cite{CL}, we know
\[
\sum_{B\in\mathcal{G}(r)}
{\frac{\sub{(A,B)}}{\lv\Aut{B}\rv}}
=\sum_{B\in\mathcal{G}(r)}
{\frac{w_r(B)}{(q)_r}\sub{(A,B)}}.
\]
Next, we will simplify this summation by making two complementary observations concerning any $B\in\mathcal{G}$.
First, if $\rank{B}>r$, then $w_r(B)=0$ by the definition of the $r$-weight of $G$.
Second, if $\rank{B}<r$, then $\sub{(A,B)}=0$.
Thus, our sum becomes
\begin{align*}
\sum_{B\in\mathcal{G}(r)}
{\frac{\sub{(A,B)}}{\lv\Aut{B}\rv}}
&=\frac{1}{(q)_r}
\cdot\sum_{B\in\mathcal{G}}
{w_r(B)\sub{(A,B)}}\\
&=\frac{1}{(q)_r}
\cdot\sum_{i=0}^{\infty}
{\sum_{\substack{
B\in\mathcal{G}\\
\lv B\rv=\ell^i}}
{w_r(B)\sub{(A,B)}}}.
\end{align*}
By Proposition~4.1 of~\cite{CL}, if $\ell^i\geq\lv A\rv$, then
\[
\sum_{\substack{
B\in\mathcal{G}\\
\lv B\rv=\ell^i}}
{w_r(B)\sub{(A,B)}}
=w_r\lp\ell^i\lv A\rv^{-1}\rp w_r(A).
\]
If, on the other hand, $\ell^i<\lv A\rv$, then the above sum clearly vanishes.
Thus,
\[
\sum_{B\in\mathcal{G}(r)}
{\frac{\sub{(A,B)}}{\lv\Aut{B}\rv}}
=\frac{1}{(q)_r}
\cdot w_r(A)
\cdot\sum_{i=0}^{\infty}{w_r\lp\ell^i\rp}.
\]
Finally, we apply Propsition~3.1 (again) and Corollary~3.7 of~\cite{CL} to obtain
\[
\sum_{B\in\mathcal{G}(r)}
{\frac{\sub{(A,B)}}{\lv\Aut{B}\rv}}
=\frac{1}{(q)_r}
\cdot\frac{(q)_r}{\lv\Aut{A}\rv}
\cdot\frac{1}{(q)_r},
\]
as desired.
\end{proof}

Before our next result, we need a few more definitions.

\begin{Not71} \label{Not71}
For any $A\in\mathcal{G}$ and any $A$-chain $\mathfrak{C}=\lb A_j\rb_{j=1}^i$, define
\[
\sub_+{\lp\mathfrak{C}\rp}
:=\sub{\lp A,A_1\rp}\prod_{j=1}^{i-1}{\sub{\lp A_j,A_{j+1}\rp}}.
\]
If $A,B\in\mathcal{G}$, let
\[
S_+(A,B)
:=\begin{cases}
\hspace{18px}0&\text{if }A\nleq B,\\
\hspace{18px}1&\text{if }A=B,\\
\displaystyle{
\sum_{\substack{A\text{-chains }\mathfrak{C},\\
\max{\mathfrak{C}}=B}}
{\sub_+{\lp\mathfrak{C}\rp}}}&\text{if } A<B.
\end{cases}
\]
Finally, suppose $A\in\mathcal{G}(r)$ for some $r\in\Z^{\geq0}$.
Then for any $i\in\Z^{\geq0}$, let
\[
\varsigma_i(A)
:=\sum_{\substack{
A\text{-chains }\mathfrak{C}\\
\max{\mathfrak{C}}\in\mathcal{G}(r)\\
\lv\mathfrak{C}\rv=i}}
{\frac{\sub_+{\lp\mathfrak{C}\rp}}
{\lv\Aut{\lp\max{\mathfrak{C}}\rp}\rv}}.
\]
We know that $\varsigma_i(A)$ converges by the following lemma.
\end{Not71}

\begin{subg5} \label{subg5}
If $A\in\mathcal{G}(r)$ for some $r\in\Z^{\geq0}$, and $i\in\Z^{\geq0}$, then
\[
\varsigma_i(A)=\frac{\lp(q)_r^{\phantom{l}-1}-1\rp^i}{\lv\Aut{A}\rv}.
\]
\end{subg5}
\begin{proof}
We will induct on $i$.
Note that the case $i=0$ is trivially true and the case $i=1$ is true by \hyperref[subg4]{Lemma~\ref*{subg4}}.
Using the induction hypothesis and \hyperref[subg4]{Lemma~\ref*{subg4}}, we see that for any $i\geq1$,
\begin{align*}
\varsigma_i(A)
&=\sum_{\substack{
A\text{-chains }\mathfrak{C}\\
\max{\mathfrak{C}}\in\mathcal{G}(r)\\
\lv\mathfrak{C}\rv=i}}
{\frac{\sub_+{\lp\mathfrak{C}\rp}}
{\lv\Aut{\lp\max{\mathfrak{C}}\rp}\rv}}\\
&=\sum_{\substack{
A\text{-chains }\mathfrak{C}\\
\max{\mathfrak{C}}\in\mathcal{G}(r)\\
\lv\mathfrak{C}\rv=i-1}}
{\sub_+{\lp\mathfrak{C}\rp}
\cdot\sum_{\substack{
B\in\mathcal{G}(r)\\
B\neq\max{\mathfrak{C}}}}
{\frac{\sub{\lp\max{\mathfrak{C}},B\rp}}{\lv\Aut{B}\rv}}}\\
&=\sum_{\substack{
A\text{-chains }\mathfrak{C}\\
\max{\mathfrak{C}}\in\mathcal{G}(r)\\
\lv\mathfrak{C}\rv=i-1}}
{\sub_+{\lp\mathfrak{C}\rp}
\cdot\frac{(q)_r^{\phantom{l}-1}-1}{\lv\Aut{\lp\max{\mathfrak{C}}\rp}\rv}}\\
&=\frac{\lp(q)_r^{\phantom{l}-1}-1\rp^i}
{\lv\Aut{A}\rv},
\end{align*}
as desired.
\end{proof}

\begin{converge3} \label{converge3}
If $A\in\mathcal{G}(r)$ for some $r\in\Z^{\geq0}$, then
\[
\sum_{B\in\mathcal{G}(r)}
{\frac{S(A,B)}{\lv\Aut{B}\rv}}
=\frac{(q)_r}{\lv\Aut{A}\rv}.
\]
\end{converge3}
\begin{proof}
Since $\ell$ is odd, we know that $0<(q)_r^{\phantom{l}-1}-1<1$.
Moreover,
\[
\sum_{B\in\mathcal{G}(r)}
{\frac{S(A,B)}{\lv\Aut{B}\rv}}
=\sum_{i=0}^{\infty}{(-1)^i\varsigma_i(A)}.
\]
\end{proof}

\begin{converge4} \label{converge4}
f $A\in\mathcal{G}(r)$ for some $r\in\Z^{\geq0}$, then
\[
\sum_{B\in\mathcal{G}(r)}
{\frac{S_+(A,B)}{\lv\Aut{B}\rv}}
=\frac{(q)_r}{\lp2(q)_r-1\rp\lv\Aut{A}\rv}.
\]
\end{converge4}

\section{Two $q$-series identites} \label{qprelim}

We record here two $q$-series identities that we will need in \hyperref[prelim2]{Section~\ref*{prelim2}}.
The first follows from the $q$-binomial theorem, and the second follows from the $q$-Chu-Vandermonde sum.
In this section, we allow $q$ to be any complex number of modulus less than 1.

\begin{q1} \label{q1}
Suppose that $k\in\Z^{>0}$.
Then
\[
\sum_{i=0}^k{(-1)^i\binom{k}{i}_qq^{\frac{i(i+1)}{2}-ik}}=0.
\]
\end{q1}
\begin{proof}
We use the following corollary of the $q$-binomial theorem: for any $z\in\C$,
\[
\sum_{i=0}^k{(-1)^i\binom{k}{i}_qq^{\frac{i(i-1)}{2}}z^i}=(z;q)_k.
\]
This is given as Corollary~10.2.2 (part~$(c)$) in~\cite{AAR} (though, as the authors mention, it is due to Rothe).
We substitute $x=-q^{-k+1}$ to obtain the result.
\end{proof}

\begin{qme} \label{qme}
Suppose that $m,k\in\Z^{\geq0}$.
Then
\[
\sum_{i=0}^{\infty}{\binom{m}{i}_q\binom{k}{i}_q\lp q\rp_iq^{(m-i)(k-i)}}=1.
\]
\end{qme}
\begin{proof}
Using the standard notation for hypergeometric functions, note that
\[
\sum_{i=0}^{\infty}{\binom{m}{i}_q\binom{k}{i}_q\lp q\rp_iq^{(m-i)(k-i)}}
=q^{mk}\cdot{}_2\phi_1\lp q^{-m},q^{-k};0;q,q\rp.
\]
The result then follows from the $q$-Chu-Vandermonde sum, which states that
\[
{}_2\phi_1\lp a,q^{-n};c;q,q\rp
=\frac{\lp\frac{c}{a};q\rp_n}{\lp c;q\rp_n}a^n;
\]
see Formula~II.6 from Appendix~II in~\cite{GR}.
\end{proof}

\section{More identities from finite abelian group theory} \label{prelim2}

Before continuing, we need one more bit of notation and a lemma that follows from work in~\cite{GMob}.
In this section, we once again set $q=1/\ell$.

\begin{Not23} \label{Not23}
For any $A\in\mathcal{G}$ and any $i\in\Z^{\geq0}$, let
\[
A_{\oplus i}
:=A\oplus\overbrace{\lp\Z/\ell\rp\oplus\cdots\oplus\lp\Z/\ell\rp}^{i\text{ times}}.
\]
\end{Not23}

\begin{autid} \label{autid}
If $s,i\in\Z^{\geq0}$ and $A\in\mathcal{G}(s,s)$, then
\[
\lv\Aut{A_{\bigoplus i}}\rv=q^{-i(2s+i)}(q)_i\lv\Aut{A}\rv.
\]
\end{autid}
\begin{proof}
Computation following from Theorem 2.2 in~\cite{GMob}.
\end{proof}

\begin{ell2} \label{ell2}
Suppose $r,s\in\Z^{\geq0}$ and $r\geq s$.
If $A\in\mathcal{G}(r,s)$ and $i\in\lb s,\ldots,r\rb$, then
\[
\sum_{B\in\mathcal{G}(r,i)}
{\frac{\sub{(A,B)}}{\lv\Aut{B}\rv}}
=\binom{r-s}{r-i}_q\cdot\frac{q^{i(i-s)}}
{(q)_i\lv\Aut{A}\rv}.
\]
\end{ell2}
\begin{proof}
We begin by considering the $i=r$ case, which we will prove by induction on $r-s$.
(When $r-s=0$, the $i=r$ case is clearly true by \hyperref[subg4]{Lemma~\ref*{subg4}}.)
Now, there exists a unique $A_0\in\mathcal{G}(s,s)$ such that $A=\lp A_0\rp_{\oplus(r-s)}$.
Next, by \hyperref[subg4]{Lemma~\ref*{subg4}}, \hyperref[autid]{Lemma~\ref*{autid}} (twice), Lemma~2.5 of~\cite{GMob}, and the induction hypothesis, we see that
\begin{align*}
\frac{1}{(q)_r\lv\Aut{A}\rv}
&=\sum_{B\in\mathcal{G}(r)}
{\frac{\sub{(A,B)}}{\lv\Aut{B}\rv}}\\
&=\sum_{j=s}^{r}{\sum_{B\in\mathcal{G}(r,j)}
{\frac{\sub{(A,B)}}{\lv\Aut{B}\rv}}}\\
&=\sum_{j=s}^{r}{\sum_{B\in\mathcal{G}(j,j)}
{\frac{q^{(r-j)(r+j)}\sub{\lp\lp A_0\rp_{\bigoplus(j-s)},B\rp}}
{(q)_{r-j}\lv\Aut{B}\rv}}}\\
&=\sum_{j=s}^{r-1}
{\frac{q^{(r-j)(r+j)}\cdot q^{j(j-s)}}
{(q)_{r-j}\cdot(q)_j\lv\Aut{\lp A_0\rp_{\bigoplus(j-s)}}\rv}}
+\sum_{B\in\mathcal{G}(r,r)}
{\frac{\sub{(A,B)}}{\lv\Aut{B}\rv}}\\
&=\sum_{j=s}^{r-1}
{\frac{q^{(r-j)(r+j)}\cdot q^{j(j-s)}\cdot q^{(j-s)(j+s)}(q)_{r-s}}
{(q)_{r-j}
\cdot(q)_j
\cdot q^{(r-s)(r+s)}(q)_{j-s}\lv\Aut{A}\rv}}
+\sum_{B\in\mathcal{G}(r,r)}
{\frac{\sub{(A,B)}}{\lv\Aut{B}\rv}}\\
&=\frac{1}{\lv\Aut{A}\rv}
\cdot\sum_{j=s}^{r-1}
{\binom{r-s}{r-j}_q
\cdot\frac{q^{j(j-s)}}{(q)_j}}
+\sum_{B\in\mathcal{G}(r,r)}
{\frac{\sub{(A,B)}}{\lv\Aut{B}\rv}}
\end{align*}
Letting $m=r-s$, we obtain
\begin{align*}
\frac{1}{(q)_r\lv\Aut{A}\rv}
&=\frac{1}{\lv\Aut{A}\rv}
\cdot\sum_{j=0}^{m-1}
{\binom{m}{m-j}_q
\cdot\frac{q^{j(r+j-m)}}{(q)_{r+j-m}}}
+\sum_{B\in\mathcal{G}(r,r)}
{\frac{\sub{(A,B)}}{\lv\Aut{B}\rv}}.
\end{align*}
Thus,
\begin{align*}
\sum_{B\in\mathcal{G}(r,r)}
{\frac{\sub{(A,B)}}{\lv\Aut{B}\rv}}
&=\frac{1}{(q)_r\lv\Aut{A}\rv}
\cdot\lp1-\sum_{j=0}^{m-1}
{\binom{m}{m-j}_q
\cdot\frac{q^{j(r+j-m)}(q)_r}{(q)_{r+j-m}}}\rp&&\\
&=\frac{1}{(q)_r\lv\Aut{A}\rv}
\cdot\lp1-\sum_{j=0}^{m-1}
{\binom{m}{j}_q\binom{r}{m-j}_q
\cdot(q)_{m-j}q^{j(r+j-m)}}\rp&&\\
&=\frac{1}{(q)_r\lv\Aut{A}\rv}
\cdot\lp1-\sum_{j=1}^{m}
{\binom{m}{j}_q\binom{r}{j}_q\cdot(q)_jq^{(m-j)(r-j)}}\rp&&\\
&=\frac{1}{(q)_r\lv\Aut{A}\rv}
\cdot\lp1-\lp1-q^{mr}\rp\rp&&
\hspace{-80px}\text{(by \hyperref[qme]{Corollary~\ref*{qme}})}\\
&=\frac{q^{r(r-s)}}{(q)_r\lv\Aut{A}\rv},
\end{align*}
concluding the $i=r$ case.

For the general case, the same techniques as above show that
\begin{align*}
\sum_{B\in\mathcal{G}(r,i)}
{\frac{\sub{(A,B)}}{\lv\Aut{A}\rv}}
&=\sum_{B\in\mathcal{G}(i,i)}
{\frac{q^{(r-i)(r+i)}\sub{\lp\lp A_0\rp_{\bigoplus(i-s)},B\rp}}
{(q)_{r-i}\lv\Aut{B}\rv}}\\
&=\frac{q^{(r-i)(r+i)}\cdot q^{i(i-s)}}
{(q)_{r-i}\cdot(q)_i\lv\Aut{\lp A_0\rp_{\bigoplus(i-s)}}\rv}\\
&=\frac{q^{(r-i)(r+i)}\cdot q^{i(i-s)}\cdot q^{(i-s)(i+s)}(q)_{r-s}}
{(q)_{r-i}\cdot(q)_i\cdot q^{(r-s)(r+s)}(q)_{i-s}\lv\Aut{A}\rv},
\end{align*}
completing the proof.
\end{proof}

\begin{Sagain} \label{Sagain}
Suppose $r,s\in\Z^{\geq0}$ and $r\geq s$.
If $A\in\mathcal{G}(r,s)$ and $i\in\lb s,\ldots,r\rb$, then
\[
\sum_{B\in\mathcal{G}(r,i)}
{\frac{S(A,B)}{\lv\Aut{B}\rv}}
=(-1)^{i-s}q^{\frac{i(i+1)}{2}-\frac{s(s+1)}{2}}
\cdot\binom{r-s}{r-i}_q
\cdot\frac{(q)_s}{\lv\Aut{A}\rv}.
\]
\end{Sagain}
\begin{proof}
Throughout this proof, we will rearrange terms of infinite alternating sums.
The fact that we can do so follows from \hyperref[converge4]{Corollary~\ref*{converge4}}.
To ease notation, let
\[
X_i:=\sum_{B\in\mathcal{G}(r,i)}
{\frac{S(A,B)}{\lv\Aut{B}\rv}}.
\]
As a first step toward the the result, we claim that
\[
X_i=
\begin{cases}
\displaystyle{\frac{(q)_s}{\lv\Aut{A}\rv}}&\text{if }i=s\\
&\vspace{-5px}\\
\displaystyle{-\sum_{j=s}^{i-1}{q^{i(i-j)}\cdot\binom{r-j}{r-i}_q\cdot X_j}}&\text{if }i>s.
\end{cases}
\]

If $i=s$, we use \hyperref[subg4]{Lemma~\ref*{subg4}} to obtain:
\begin{align*}
X_s
&=\sum_{j=0}^\infty{\sum_{\substack{
A\text{-chains }\mathfrak{C}\\
\max{\mathfrak{C}}\in\mathcal{G}(r,s)\\
\lv\mathfrak{C}\rv=j}}
{\lp\frac{\sub{\lp\mathfrak{C}\rp}}{\lv\Aut{\lp\max{\mathfrak{C}}\rp}\rv}\rp}}\\
&=\frac{1}{\lv\Aut{A}\rv}
-\sum_{j=1}^\infty{\sum_{\substack{
A\text{-chains }\mathfrak{C}\\
\max{\mathfrak{C}}\in\mathcal{G}(r,s)\\
\lv\mathfrak{C}\rv=j-1}}
{\sub{\lp\mathfrak{C}\rp}}
\cdot\sum_{\substack{
B\in\mathcal{G}(r,s)\\
B\neq\max{\mathfrak{C}}}}
{\lp\frac{\sub{\lp\max{\mathfrak{C}},B\rp}}{\lv\Aut{B}\rv}\rp}}\\
&=\frac{1}{\lv\Aut{A}\rv}
-\lp\frac{1}{(q)_s}-1\rp\sum_{j=1}^\infty{\sum_{\substack{
A\text{-chains }\mathfrak{C}\\
\max{\mathfrak{C}}\in\mathcal{G}(r,s)\\
\lv\mathfrak{C}\rv=j-1}}
{\frac{\sub{\lp\mathfrak{C}\rp}}{\lv\Aut{\lp\max{\mathfrak{C}}\rp}\rv}}}\\
&=\frac{1}{\lv\Aut{A}\rv}
-\lp\frac{1}{(q)_s}-1\rp X_s.
\end{align*}
The result follows by solving for $X_s$.
If, on the other hand, $i>s$, we can apply \hyperref[ell2]{Lemma~\ref*{ell2}} to see that
\begin{align*}
X_i
&=\sum_{k=1}^\infty{\sum_{\substack{
A\text{-chains }\mathfrak{C}\\
\max{\mathfrak{C}}\in\mathcal{G}(r,i)\\
\lv\mathfrak{C}\rv=k}}
{\lp\frac{\sub{\lp\mathfrak{C}\rp}}{\lv\Aut{\lp\max{\mathfrak{C}}\rp}\rv}\rp}}\\
&=-\sum_{B\in\mathcal{G}(r,i)}{\frac{\sub{\lp A,B\rp}}{\lv\Aut{B}\rv}}
-\sum_{j=s}^{i-1}{\sum_{k=2}^\infty{\sum_{\substack{
A\text{-chains }\mathfrak{C}\\
\max{\mathfrak{C}}\in\mathcal{G}(r,j)\\
\lv\mathfrak{C}\rv=k-1}}
{\sub{(\mathfrak{C})}}
\cdot\sum_{B\in\mathcal{G}(r,i)}
{\frac{\sub{\lp\max{\mathfrak{C}},B\rp}}{\lv\Aut{B}\rv}}}}\\
&\hspace{20px}-\sum_{k=2}^\infty{\sum_{\substack{
A\text{-chains }\mathfrak{C}\\
\max{\mathfrak{C}}\in\mathcal{G}(r,i)\\
\lv\mathfrak{C}\rv=k-1}}
{\sub{(\mathfrak{C})}}
\cdot\sum_{\substack{B\in\mathcal{G}(r,i)\\B\neq\max{\mathfrak{C}}}}
{\frac{\sub{\lp\max{\mathfrak{C}},B\rp}}{\lv\Aut{B}\rv}}}\\
&=-\binom{r-s}{r-i}_q
\cdot\frac{q^{i(i-s)}}{(q)_i\lv\Aut{A}\rv}
-\sum_{j=s}^{i-1}{\binom{r-j}{r-i}_q\cdot\frac{q^{i(i-j)}}{(q)_i}
\cdot\sum_{k=2}^\infty{\sum_{\substack{
A\text{-chains }\mathfrak{C}\\
\max{\mathfrak{C}}\in\mathcal{G}(r,k)\\
\lv\mathfrak{C}\rv=k-1}}
{\frac{\sub{(\mathfrak{C})}}{\lv\Aut{\lp\max{\mathfrak{C}}\rp}\rv}}}}\\
&\hspace{20px}-\lp\frac{1}{(q)_i}-1\rp
\sum_{k=2}^\infty{\sum_{\substack{
A\text{-chains }\mathfrak{C}\\
\max{\mathfrak{C}}\in\mathcal{G}(r,k)\\
\lv\mathfrak{C}\rv=k-1}}
{\frac{\sub{(\mathfrak{C})}}{\lv\Aut{\lp\max{\mathfrak{C}}\rp}\rv}}}\\
&=-\sum_{j=s}^{i-1}
{\lp\binom{r-j}{r-i}_q\cdot\frac{q^{i(i-j)}}{(q)_i}\cdot X_j\rp}
-\lp\frac{1}{(q)_i}-1\rp X_i.
\end{align*}
The claim follows by solving for $X_i$.

By the claim above, the lemma is true for $i=s$; we will now prove the lemma by induction on $i$.
Using the induction hypothesis, note that
\begin{align*}
X_i
&=-\sum_{j=s}^{i-1}{q^{i(i-j)}\cdot\binom{r-j}{r-i}_q\cdot X_j}\\
&=-\sum_{j=s}^{i-1}
{\frac{q^{i(i-j)}}{(q)_{r-i}(q)_{i-j}}
\cdot\frac{(-1)^{j-s}q^{\frac{j(j+1)}{2}}(q)_{r-s}(q)_s}
{q^{\frac{s(s+1)}{2}}(q)_{j-s}\lv\Aut{A}\rv}}\\
&=-\frac{(q)_{r-s}(q)_s}{q^{\frac{s(s+1)}{2}}(q)_{r-i}\lv\Aut{A}\rv}
\cdot\sum_{j=s}^{i-1}{\frac{(-1)^{j-s}q^{i(i-j)+\frac{j(j+1)}{2}}}
{(q)_{i-j}(q)_{j-s}}}\\
&=-\frac{(q)_{r-s}(q)_s}{q^{\frac{s(s+1)}{2}}(q)_{r-i}\lv\Aut{A}\rv}
\cdot\sum_{j=0}^{i-s-1}{\frac{(-1)^jq^{i(i-j-s)+\frac{(j+s)(j+s+1)}{2}}}
{(q)_{i-j-s}(q)_j}}.
\end{align*}
Letting $k=i-s$, and using \hyperref[q1]{Corollary~\ref*{q1}}, we see that
\begin{align*}
X_i
&=-\frac{q^{ik}(q)_{r-s}(q)_s}{(q)_{r-i}\lv\Aut{A}\rv}
\cdot\sum_{j=0}^{k-1}{(-1)^j\frac{q^{\frac{j(j+1)}{2}-jk}}{(q)_{k-j}(q)_j}}\\
&=-\frac{q^{ik}(q)_{r-s}(q)_s}{(q)_{r-i}\lv\Aut{A}\rv}
\cdot\lp-(-1)^k\frac{q^{\frac{k(k+1)}{2}-k^2}}{(q)_k}\rp.
\end{align*}
Substituting for $k$ gives the result.
\end{proof}

\bibliography{abelianqseries}
\bibliographystyle{amsalpha}

\end{document}